\documentclass[12pt,a4paper]{article}
\usepackage{amsmath,amssymb,amsthm}
\usepackage{ascmac}
\usepackage{bm}

\newtheorem{thm}{Theorem}[section]
\newtheorem{de}[thm]{Definition}
\newtheorem{ex}[thm]{Example}
\newtheorem{prp}[thm]{Proposition}
\newtheorem{cor}[thm]{Corollary}
\newtheorem{lem}[thm]{Lemma}
\newtheorem{rem}[thm]{Remark}

 \def\Z{\mathbb {Z}} \def\Q{\mathbb {Q}}
 
\def\za{\zeta_{\mathcal{A}}}
\def\AA{\mathcal{A}}

\def\fh{\frak H} \def\fho{\fh^1}  
\def\kk{\textbf{k}} \def\ee{\textbf{e}}  \def\ch{\vee}
\font\sevency=wncyr7  \def\sh{\, \hbox{\sevency X}\, }

\title{\LARGE{Ohno-type relation for finite multiple zeta values}}
\author{\large{Kojiro Oyama}}
\date{}
\begin{document}
\maketitle

\pagenumbering{arabic}

\begin{abstract}
Ohno's relation is a well-known relation among multiple zeta values.
In this paper, we prove Ohno-type relation for finite multiple zeta values, which is conjectured by Kaneko.
As a corollary, we give an alternative proof of the sum formula for finite multiple zeta values, which was first proved by Saito and Wakabayashi.
\end{abstract}


\section{Introduction}
\subsection{Finite multiple zeta values}
The finite multiple zeta value is defined by Zagier (\cite{kaneko,kaneko2}).

\begin{de}\label{A}
The $\Q$-algebra $\AA$ is defind by
\begin{equation*}
\AA:=\left(\prod_{p:\rm{prime}} \Z/p\Z\right)\Biggm/\left(\bigoplus_{p:\rm{prime}} \Z/p\Z\right) = \{(a_p)_p~|~a_p\in \Z/p\Z\}/{\sim} ,
\end{equation*}
where $(a_p)_p\sim(b_p)_p$ means $a_p=b_p$ for all but finitely many primes p.
\end{de}

\begin{de}\label{fmzvs}For any \,$\textnormal{\kk}=(k_1,\ldots,k_r)\in(\Z_{\ge1})^r$, the finite multiple zeta value is defined by
\begin{equation*}
\za(\textnormal{\kk})=\za(k_1,\ldots,k_r):=\left(\sum_{\substack{p>m_1>\cdots>m_r>0}}\frac{1}{m_1^{k_1}\cdots m_r^{k_r}}\bmod p\right)_p \in \AA.
\end{equation*}
\end{de}

For an index $\kk=(k_1,\ldots,k_r)\in(\Z_{\geq0})^r$, the integer $k_1+\cdots+k_r$ is called the weight of $\kk$ (denoted by $\textnormal{wt(\kk)}$) 
and the integer $r$ is called the depth of $\kk$ (denoted by $\textnormal{dep(\kk)}$). For $\kk=(k_1,\ldots.k_r)$ and $\kk'=(k_1',\ldots,k_r')$, the symbol $\textnormal{\kk}+\textnormal{\kk}'$ represents the componentwise sum:
\begin{equation*}
\textnormal{\kk}+\textnormal{\kk}':=(k_1+k_1',\ldots,k_r+k_r').
\end{equation*}
For any $\textnormal{\kk}=(k_1,\ldots,k_r)$, we define Hoffman's dual of $\textnormal{\kk}$ by
\begin{equation*}
\textnormal{\kk}^{\ch}:=(\underbrace{1,\ldots,1}_{k_1}+\underbrace{1,\ldots,1}_{k_2}+1,\ldots,1+\underbrace{1,\ldots,1}_{k_r}).
\end{equation*}
By definition, we easily find
\begin{equation*}
\textnormal{dep(\kk)}+\textnormal{dep}(\textnormal{\kk}^{\ch})=\textnormal{wt(\kk)}+1,\quad
(\kk^{\ch})^{\ch}=\kk.
\end{equation*}

\begin{ex}
\begin{align*}
(2,3,1,2)^{\ch}&=(\underbrace{1,1}_{2}+\underbrace{1,1,1}_{3}+\underbrace{1}_{1}+\underbrace{1,1}_{2})\\
&=(1,2,1,3,1),\\
(1,2,1,3,1)^{\ch}&=(\underbrace{1}_{1}+\underbrace{1,1}_{2}
+\underbrace{1}_{1}+\underbrace{1,1,1}_{3}+\underbrace{1}_{1})\\
&=(2,3,1,2).
\end{align*}
\end{ex}

Our main theorem is the following result conjectured by Kaneko (\cite{kaneko,kaneko2}):
\begin{thm}[Main Theorem]\label{main}
For any $(k_1,\ldots,k_r)\in(\Z_{\ge1})^r$ and $n\in\Z_{\geq0}$, we have
\begin{equation*}
\sum_{\substack{e_1+\cdots+e_r=n \\
^\forall e_i\geq 0}}
\za(k_1+e_1,\ldots,k_r+e_r)
=
\sum_{\substack{e'_1+\cdots+e'_s=n \\
^\forall e'_i\geq 0}}
\za((k'_1+e'_1,\ldots,k'_s+e'_s)^{\ch}),
\end{equation*}
where $(k'_1,\ldots,k'_s)=(k_1,\ldots,k_r)^{\ch}$ is Hoffman's dual of $(k_1,\ldots,k_r)\in(\Z_{\ge1})^r$.
\end{thm}

\begin{rem}
We can prove Ohno-type relation for finite real multiple zeta values, by exactly the same proof given in $\S$2.
\end{rem}

\subsection{Algebraic setup and Notation}

Let $\fh:=\Q\langle x,y\rangle$ be the non-commutative polynomial algebra over $\Q$ in two indeterminates $x$ and $y$, 
and $\fho:=\Q+ \fh y$ its subalgebra. Put $z_k:=x^{k-1}y$ for a positive integer $k$. Then $\fho$ is generated by $z_k\;(k=1,2,3,\ldots)$.
For a word $w\in\fh$, the total degree of $w$ is called the weight of $w$ (denoted by $|w|$) and the degree in $y$ is called the depth of $w$.

The harmonic product $*$ on $\fho$ is the $\Q$-bilinear map $*:\fho\longrightarrow\fho$ defined inductively by the following rules:\\
{\bf(i)} 
For any $w\in\fho$, $w*1=1*w=w$.\\
{\bf(ii)} 
For any $w_1,w_2\in\fho$ and $m,n\in\Z_{\geq1}$,
\begin{equation*}
z_mw_1*z_nw_2=z_m(w_1*z_nw_2)+z_n(z_mw_1*w_2)+z_{m+n}(w_1*w_2).
\end{equation*}

The shuffle product $\sh$ on $\fh$ is the $\Q$-bilinear map $\sh:\fh\longrightarrow\fh$ defined inductively by the following rules:\\
{\bf(i)} For any $w\in\fh$, $w\sh1=1\sh w=w$.\\
{\bf(ii)} For $u_1,u_2\in\{x,y\}$ and any $w_1,w_2\in\fh$,
\begin{equation*}
(u_1w_1)\sh (u_2w_2)=u_1(w_1\sh u_2w_2)+u_2(u_1w_1\sh w_2).
\end{equation*}
It is known that these products are commutative and associative (cf: \cite{hoffman3,reu}).

We define a $\Q$-linear map $Z_\AA:\fho\longrightarrow \AA$ by setting
\begin{equation*}
Z_\AA(1)=(1)_p,\quad Z_\AA(z_{k_1}\cdots z_{k_r})=\za(k_1,\ldots,k_r).
\end{equation*}

The following  proposition gives the important properties of $Z_\AA$:

\begin{prp}[\cite{hoffman3,kaneko,kaneko2}]\label{qmap}For any words $w=z_{k_1}\cdots z_{k_r}$ and $w'=z_{k'_1}\cdots z_{k'_s}\in\fho$, we have
\begin{eqnarray*}
&(1)&Z_\AA(w*w')=Z_\AA(w)Z_\AA(w'),\\
&(2)&Z_\AA(w\sh w')=(-1)^{|w|}Z_\AA(z_{k_r}\cdots z_{k_1}z_{k'_1}\cdots z_{k'_s}).
\end{eqnarray*}
\end{prp}

We define a $\Q$-linear map $T:\fho\longrightarrow\fho$ by
\begin{equation*}
T(w):=\tau(w')y
\end{equation*}
for any $w=w'y\in\fho$, where $\tau$ is the automorphism of $\fh$ given by
\begin{equation*}
\tau(x)=y,\quad\tau(y)=x.
\end{equation*}
The map $T$ represents Hoffman's dual for words, namely, if $Z_\AA(w)=\za(\textnormal{\kk})$, then $Z_\AA(T(w))=\za(\textnormal{\kk}^\ch)$.


\section{Proof of the main theorem}
\subsection{Preliminaries}
\begin{prp}[{\cite[Corollary 3]{ikz}}]\label{IKZ}
For any $w\in\fho$, we have
\begin{equation}\label{lem2}
\frac{1}{1-yu}*w=\frac{1}{1-yu}\sh\Delta_u(w),
\end{equation}
where $\Delta_u$ is the automorphism of ${\fh}$ given by
\begin{equation*}
\Delta_u(x)=x(1+yu)^{-1},\quad\Delta_u(y)=y+x(1+yu)^{-1}yu,
\end{equation*}
and $u$ is a formal parameter.
\end{prp}

We define the symbols required to prove Theorem \ref{main} by the following:

\begin{de}For $\textnormal{\kk}=(k_1,k_2,\ldots,k_r)\in(\Z_{\geq1})^r,r\in\Z_{\geq1}$, and $n,l\in\Z_{\geq0}$, we define
\begin{align*}
a_l(k_1,k_2,\ldots,k_r)&:=\sum_{\substack{e_1+\cdots+e_{k_1+\cdots+k_r}=l\\^\forall e_i\geq0}}
\left(\prod_{j_1=1}^{k_1}xy^{e_{j_1}}\right)y\left(\prod_{j_2=1}^{k_2}xy^{e_{k_1+j_2}}\right)y\cdots
\left(\prod_{j_r=1}^{k_r}xy^{e_{k_1+\cdots+k_{r-1}+j_r}}\right)y,\\
A_{\textnormal{\kk},l,n}&:=\sum_{\substack{\lambda_1+\cdots+\lambda_r=n\\^\forall\lambda_j\in\{0,1\}}}a_l(k_1-1+\lambda_1,k_2-1+\lambda_2,\ldots,k_r-1+\lambda_r).
\end{align*}
\end{de}

\begin{lem}For any $\textnormal{\kk}=(k_1,k_2,\ldots,k_r)$ and $n\in\Z_{\geq1}$, we have
\begin{equation}\label{lem1}
\sum_{i=0}^{\min\{n,r\}}\left((-1)^{i}\sum_{\substack{k+l=n-i\\ k,l\geq0}}Z_\AA\left(y^kA_{\textnormal{\kk},l,i}\right)\right)=(0)_p.
\end{equation}
\end{lem}

\begin{proof}
For $w=z_{k_1}z_{k_2}\cdots z_{k_r}$, we calculate both sides of \eqref{lem2}.
\begin{equation*}
\textnormal{L.H.S. of \eqref{lem2}}=\sum_{n=0}^{\infty}(yu)^n*z_{k_1}z_{k_2}\cdots z_{k_r}=\sum_{n=0}^{\infty}(z_1^n*z_{k_1}z_{k_2}\cdots z_{k_r})u^n.
\end{equation*}
Here, by the definition of $\Delta_u$, we obtain 
\begin{align*}
\Delta_u(z_{k_1}z_{k_2}\cdots z_{k_r})=&\underbrace{x(1+yu)^{-1}\cdots x(1+yu)^{-1}}_{k_1-1}\{y+x(1+yu)^{-1}yu\}\\
&\times \underbrace{x(1+yu)^{-1}\cdots x(1+yu)^{-1}}_{k_2-1}\{y+x(1+yu)^{-1}yu\}\\
&\times \cdots\\
&\times \underbrace{x(1+yu)^{-1}\cdots x(1+yu)^{-1}}_{k_r-1}\{y+x(1+yu)^{-1}yu\}\\
=&\sum_{l=0}^{\infty}\sum_{i=0}^r(-1)^lA_{\textnormal{\kk},l,i}u^{l+i}.
\end{align*} 
Thus we have
\begin{align*}
\textnormal{R.H.S. of \eqref{lem2}}&=\left(\sum_{k=0}^{\infty}(yu)^k \right)\sh\left(\sum_{l=0}^{\infty}\sum_{i=0}^r (-1)^l A_{\textnormal{\kk},l,i}u^{l+i}\right)\\
&=\sum_{n=0}^{\infty}\sum_{i=0}^r\left(\sum_{\substack{k+l=n\\ k,l\geq0}}(-1)^ly^k\sh A_{\textnormal{\kk},l,i}\right)u^{n+i}.
\end{align*}
Comparing the cofficients of $u^n$ on both sides, we have
\begin{equation}\label{lem3}
\sum_{i=0}^{\min\{n,r\}}\left(\sum_{\substack{k+l=n-i\\ k,l\geq0}}(-1)^ly^k\sh A_{\textnormal{\kk},l,i}\right)=z_1^n*z_{k_1}z_{k_2}\cdots z_{k_r}.
\end{equation}
Applying $Z_\AA$ to both sides of \eqref{lem3} gives
\begin{align*}
Z_\AA\left(\textnormal{L.H.S. of \eqref{lem3}}\right)&=
Z_\AA\left(\sum_{i=0}^{\min\{n,r\}}\left(\sum_{\substack{k+l=n-i\\ k,l\geq0}}(-1)^ly^k\sh A_{\textnormal{\kk},l,i}\right)\right)\\
&=\sum_{i=0}^{\min\{n,r\}}\left(\sum_{\substack{k+l=n-i\\ k,l\geq0}}(-1)^lZ_\AA\left(y^k\sh A_{\textnormal{\kk},l,i}\right)\right)\\
&=\sum_{i=0}^{\min\{n,r\}}\left(\sum_{\substack{k+l=n-i\\ k,l\geq0}}(-1)^{n-i}Z_\AA\left(y^kA_{\textnormal{\kk},l,i}\right)\right),\\
Z_\AA\left(\textnormal{R.H.S. of \eqref{lem3}}\right)&=Z_\AA\left(z_1^n*z_{k_1}z_{k_2}\cdots z_{k_r}\right)\\
&=Z_\AA\left(z_1^n\right)Z_\AA\left(z_{k_1}z_{k_2}\cdots z_{k_r}\right)\\
&=\za(\underbrace{1,\ldots,1}_{n})\za(k_1,k_2,\ldots,k_r)=(0)_p.
\end{align*}
The last equality is a consequence of \cite[Equation(15)]{hoffman} which asserts that $\za(\underbrace{a,\ldots,a}_{r})=(0)_p$ for any $a,r\in\Z_{\geq1}$. 
Multiplying both sides by $(-1)^n$ gives
\begin{equation*}
\sum_{i=0}^{\min\{n,r\}}\left((-1)^{i}\sum_{\substack{k+l=n-i\\ k,l\geq0}}Z_\AA\left(y^kA_{\textnormal{\kk},l,i}\right)\right)=(0)_p.\qedhere
\end{equation*}
\end{proof}

\begin{lem}\label{key}For any $\textnormal{\kk}=(k_1,k_2,\ldots,k_r)$ and $n\in\Z_{\geq1}$, we have
\begin{equation*}
\sum_{i=0}^{\min\{n,r\}}\left((-1)^i\sum_{\substack{{\bm{\lambda}} \in \{0,1\}^r\\ \textnormal{wt}(\bm{\lambda})=i}}
\sum_{\substack{\textnormal{\ee}\in(\Z_{\geq0})^{s+i}\\ \textnormal{wt}(\textnormal{\ee})=n-i}} \za\left(\left((\textnormal{\kk}+\bm{\lambda})^{\ch}+\textnormal{\ee}\right)^{\ch}\right)
\right)=(0)_p,
\end{equation*}
where $s=\textnormal{dep}(\textnormal{\kk}^{\ch})$.
\end{lem}

\begin{proof}
We first show the following equality:
\begin{equation}\label{lem4}
\sum_{\substack{k+l=n\\ k,l\geq0}}Z_\AA\left(y^k a_l(k_1-1,k_2-1,\ldots,k_r-1)\right)
=\sum_{\substack{\textnormal{\ee}\in(\Z_{\geq0})^{s}\\ \textnormal{wt}(\textnormal{\ee})=n}}\za((\textnormal{\kk}^{\ch}+ \textnormal{\ee})^{\ch}),
\end{equation}
where $s=\textnormal{dep}(\textnormal{\kk}^{\ch})$. By the definition of $a_l(k_1,k_2,\ldots,k_r)$, we easily find
\begin{align*}
\textnormal{L.H.S. of \eqref{lem4}}&=\sum_{\substack{k+e_1+\cdots+e_{k_1+\cdots+k_r-r}=n\\k,^\forall e_i\geq0}}
Z_\AA\left(y^k\left(\prod_{j_1=1}^{k_1-1}xy^{e_{j_1}}\right)y
\cdots\left(\prod_{j_r=1}^{k_r-1}xy^{e_{k_1+\cdots+k_{r-1}-r+1+j_r}}\right)y\right)\\
&=\sum_{\substack{e_1+\cdots+e_{k_1+\cdots+k_r-r+1}=n\\^\forall e_i\geq0}}
Z_\AA\left(y^{e_1}\left(\prod_{j_1=1}^{k_1-1}xy^{e_{j_1+1}}\right)y
\cdots\left(\prod_{j_r=1}^{k_r-1}xy^{e_{k_1+\cdots+k_{r-1}-r+2+j_r}}\right)y\right)\\
&=\sum_{\substack{e_1+\cdots+e_{k_1+\cdots+k_r-r+1}=n\\^\forall e_i\geq0}}\za((\textnormal{\kk}^{\ch}+ \textnormal{\ee})^{\ch}),
\end{align*}
where the last equality holds by observing the following:
The word which corresponds to the index $\kk^{\ch}=(k_1,\ldots,k_r)^{\ch}$ is 
\[T(x^{k_1-1}yx^{k_2-1}y\cdots x^{k_r-1}y)=y^{k_1-1}xy^{k_2-1}x\cdots y^{k_r},\] 
and thus \[ \left(\prod_{j_1=1}^{k_1-1}x^{e_{j_1}}y\right)x\left(\prod_{j_2=1}^{k_2-1}x^{e_{k_1-1+j_2}}y\right)x
\cdots\left(\prod_{j_r=1}^{k_r}x^{e_{k_1+\cdots+k_{r-1}-r+1+j_r}}y\right)\] corresponds to $\kk^{\ch}+\ee$ $\left(\ee=(e_1,\ldots,e_{k_1+\cdots+k_r-r+1})\right).$ 
By taking Hoffman's dual, we see that the word which corresponds to $(\kk^{\ch}+\ee)^{\ch}$ is \: \begin{align*}
&T\left(\left(\prod_{j_1=1}^{k_1-1}x^{e_{j_1}}y\right)x
\left(\prod_{j_2=1}^{k_2-1}x^{e_{k_1-1+j_2}}y\right)x
\cdots\left(\prod_{j_r=1}^{k_r}x^{e_{k_1+\cdots+k_{r-1}-r+1+j_r}}y\right)\right)\\
&=y^{e_1}\left(\prod_{j_1=1}^{k_1-1}xy^{e_{j_1+1}}\right)y
\left(\prod_{j_2=1}^{k_2-1}xy^{e_{k_1+j_2}}\right)y
\cdots\left(\prod_{j_r=1}^{k_r-1}xy^{e_{k_1+\cdots+k_{r-1}-r+2+j_r}}\right)y.
\end{align*}
Thus we find \eqref{lem4}. Accordingly, we have
\begin{align*}
&\textnormal{L.H.S. of \eqref{lem1}}\\
&=\sum_{i=0}^{\min\{n,r\}}\left((-1)^{i}\sum_{\substack{k+l=n-i\\ k,l\geq0}}Z_\AA\left(y^kA_{\textnormal{\kk},l,i}\right)\right)\\
&=\sum_{i=0}^{\min\{n,r\}}\left((-1)^{i}\sum_{\substack{k+l=n-i\\ k,l\geq0}}\sum_{\substack{\lambda_1+\cdots+\lambda_r=i\\^\forall\lambda_j\in\{0,1\}}}
Z_\AA\left(y^ka_l(k_1-1+\lambda_1,k_2-1+\lambda_2,\ldots,k_r-1+\lambda_r)\right)\right)\\
&=\sum_{i=0}^{\min\{n,r\}}\left((-1)^i\sum_{\substack{{\bm{\lambda}} \in \{0,1\}^r\\ \textnormal{wt}(\bm{\lambda})=i}}
\sum_{\substack{\textnormal{\ee}\in(\Z_{\geq0})^{s+i}\\ \textnormal{wt}(\textnormal{\ee})=n-i}} \za\left(\left((\textnormal{\kk}+\bm{\lambda})^{\ch}+\textnormal{\ee}\right)^{\ch}\right)
\right).\qedhere
\end{align*}
\end{proof}

\subsection{Proof of the main theorem}
Thorem \ref{main} can be rewritten as follows:
\begin{thm}[Theorem \ref{main}]
For any $\textnormal{\kk}=(k_1,k_2,\cdots,k_r)$ and $n\in\Z_{\geq0}$, we have
\begin{equation*}
\sum_{\substack{\textnormal{\ee}\in(\Z_{\geq0})^{r}\\ \textnormal{wt}(\textnormal{\ee})=n}}\za(\textnormal{\kk}+ \textnormal{\ee})
=\sum_{\substack{\textnormal{\ee}'\in(\Z_{\geq0})^{s}\\ \textnormal{wt}(\textnormal{\ee}')=n}}\za((\textnormal{\kk}^{\ch}+ \textnormal{\ee}')^{\ch}),
\end{equation*}
where $s=\textnormal{dep}(\textnormal{\kk}^{\ch})$.
\end{thm}

\begin{proof}Let us prove Theorem \ref{main} by induction on $n$. 
First, it is trivial in the case $n=0$. 
Next, we assume that the equality holds for values smaller than $n$. Then the assumption shows that
\begin{align*}
\sum_{\substack{{\bm{\lambda}} \in \{0,1\}^r\\ \textnormal{wt}(\bm{\lambda})=i}}
\sum_{\substack{\textnormal{\ee}\in(\Z_{\geq0})^{s+i}\\ \textnormal{wt}(\textnormal{\ee})=n-i}} \za\left(\left((\textnormal{\kk}+\bm{\lambda})^{\ch}+\textnormal{\ee}\right)^{\ch}\right)
&=\sum_{\substack{{\bm{\lambda}} \in \{0,1\}^r\\ \textnormal{wt}(\bm{\lambda})=i}}
\sum_{\substack{\textnormal{\ee}'\in(\Z_{\geq0})^{r}\\ \textnormal{wt}(\textnormal{\ee}')=n-i}} \za\left((\textnormal{\kk}+\bm{\lambda})+\textnormal{\ee}'\right)\\
&=\sum_{1\leq \mu_1<\cdots<\mu_i\leq r}\sum_{\substack{\ee''\in(\Z_{\geq0})^r\\ \textnormal{wt}(\ee')=n\\ e_{\mu_1},\ldots,e_{\mu_i}\geq1}}\za(\kk+\ee'')
\end{align*}
for $i$ with $1\leq i\leq\min\{n,r\}$. 
By Lemma\;\ref{key}, we obtain
\begin{align*}
\sum_{\substack{{\bm{\lambda}} \in \{0,1\}^r\\ \textnormal{wt}(\bm{\lambda})=0}}
\sum_{\substack{\textnormal{\ee}\in(\Z_{\geq0})^{s}\\ \textnormal{wt}(\textnormal{\ee})=n}} \za\left(\left((\textnormal{\kk}+\bm{\lambda})^{\ch}+\textnormal{\ee}\right)^{\ch}\right)
&=
\sum_{i=1}^{\min\{n,r\}}\left((-1)^{i-1}\sum_{\substack{{\bm{\lambda}} \in \{0,1\}^r\\ \textnormal{wt}(\bm{\lambda})=i}}
\sum_{\substack{\textnormal{\ee}\in(\Z_{\geq0})^{s+i}\\ \textnormal{wt}(\textnormal{\ee})=n-i}} \za\left(\left((\textnormal{\kk}+\bm{\lambda})^{\ch}+\textnormal{\ee}\right)^{\ch}\right)
\right).
\end{align*}
Therefore we have
\begin{align*}
\sum_{\substack{\textnormal{\ee}\in(\Z_{\geq0})^{s}\\ \textnormal{wt}(\textnormal{\ee})=n}} \za\left(\left(\textnormal{\kk}^{\ch}+\textnormal{\ee}\right)^{\ch}\right)
&=
\sum_{\substack{{\bm{\lambda}} \in \{0,1\}^r\\ \textnormal{wt}(\bm{\lambda})=0}}
\sum_{\substack{\textnormal{\ee}\in(\Z_{\geq0})^{s}\\ \textnormal{wt}(\textnormal{\ee})=n}} \za\left(\left((\textnormal{\kk}+\bm{\lambda})^{\ch}+\textnormal{\ee}\right)^{\ch}\right)\\
&=
\sum_{i=1}^{\min\{n,r\}}\left((-1)^{i-1}\sum_{\substack{{\bm{\lambda}} \in \{0,1\}^r\\ \textnormal{wt}(\bm{\lambda})=i}}
\sum_{\substack{\textnormal{\ee}\in(\Z_{\geq0})^{s+i}\\ \textnormal{wt}(\textnormal{\ee})=n-i}} \za\left(\left((\textnormal{\kk}+\bm{\lambda})^{\ch}+\textnormal{\ee}\right)^{\ch}\right)
\right)\\
&=\sum_{i=1}^{\min\{n,r\}}\left((-1)^{i-1}\sum_{1\leq \mu_1<\cdots<\mu_i\leq r}
\sum_{\substack{\ee''\in(\Z_{\geq0})^r\\ \textnormal{wt}(\ee'')=n\\ e_{\mu_1},\ldots,e_{\mu_i}\geq1}}\za(\kk+\ee'')
\right).
\end{align*}
Hence we must show the following identity:
\begin{equation*}
\sum_{i=1}^{\min\{n,r\}}\left((-1)^{i-1}\sum_{1\leq \mu_1<\cdots<\mu_i\leq r}
\sum_{\substack{\ee''\in(\Z_{\geq0})^r\\ \textnormal{wt}(\ee'')=n\\ e_{\mu_1},\ldots,e_{\mu_i}\geq1}}\za(\kk+\ee'')
\right)
=\sum_{\substack{\textnormal{\ee}\in(\Z_{\geq0})^{r}\\ \textnormal{wt}(\textnormal{\ee})=n}} \za(\kk+\ee).
\end{equation*}
We fix the index $\ee=(e_1,\ldots,e_r)$ such that $\textnormal{wt(\ee)}=n$, 
and let the number of non-zero components of the fixed index $\ee$ be $m$ $(1\leq m\leq\min\{n,r\})$. 
In the left-hand side, the number of appearances of the index $\ee$ is

\begin{equation*}
\binom{m}{1}-\binom{m}{2}+\binom{m}{3}-\cdots=\sum_{i=1}^{m}(-1)^{i-1}\binom{m}{i}=1.
\end{equation*}
Thus we have
\begin{equation*}
\sum_{i=1}^{\min\{n,r\}}\left((-1)^{i-1}\sum_{1\leq \mu_1<\cdots<\mu_i\leq r}
\sum_{\substack{\ee''\in(\Z_{\geq0})^r\\ \textnormal{wt}(\ee'')=n\\ e_{\mu_1},\ldots,e_{\mu_i}\geq1}}\za(\kk+\ee'')
\right)
=\sum_{\substack{\textnormal{\ee}\in(\Z_{\geq0})^{r}\\ \textnormal{wt}(\textnormal{\ee})=n}} \za(\kk+\ee).
\end{equation*}
Therefore the identity holds for $n$. The proof is complete.
\end{proof}


\section{Sum formula}
The following identity was first proved by Saito and Wakabayashi (\cite{saito}).
\begin{cor}[Sum fomula {\cite[Theorem 1.4]{saito}}]For $k,r,i\in\Z_{\geq1}$ with $1\leq i\leq r\leq k-1$, we have
\begin{equation}\label{sumformula}
\sum_{\substack{k_1+\cdots+k_r=k\\^\forall k_j\geq1,k_i\geq2}}\za(k_1,\ldots,k_r)=
(-1)^{i-1}\left(\binom{k-1}{i-1}+(-1)^r\binom{k-1}{r-i}\right)\left(\frac{B_{p-k}}{k}\bmod p\right)_p.
\end{equation}
\end{cor}

\begin{proof}
Let $\kk=(\underbrace{1,\ldots,1}_{i-1},2,\underbrace{1,\ldots,1}_{r-i})$ and $n=k-r-1$.
Then $\kk^{\ch}=(i,r+1-i)$ and $s=2$. Apply Theorem \ref{main} to this $\kk$ to get:
\begin{align*}
&\sum_{\substack{\textnormal{\ee}\in(\Z_{\geq0})^{r}\\ \textnormal{wt}(\textnormal{\ee})=k-r-1}}\za(\textnormal{\kk}+ \textnormal{\ee})\\
&=\sum_{\substack{e_1+\cdots+e_r=k-r-1\\^\forall e_j\geq0}}\za(1+e_1,\ldots,1+e_{i-1},2+e_{i},1+e_{i+1},\ldots,1+e_r)\\
&=\sum_{\substack{k_1+\cdots+k_r=k\\^\forall k_j\geq1,k_i\geq2}}\za(k_1,\ldots,k_r),\\
&\sum_{\substack{\textnormal{\ee}'\in(\Z_{\geq0})^{2}\\ \textnormal{wt}(\textnormal{\ee}')=k-r-1}}\za((\textnormal{\kk}^{\ch}+ \textnormal{\ee}')^{\ch})\\
&=\sum_{\substack{e_1'+e_2'=k-r-1\\^\forall e_j'\geq0}}\za((i+e_1',r+1-i+e_2')^{\ch})\\
&=\sum_{\substack{e_1'+e_2'=k-r-1\\^\forall e_j'\geq0}}\za(\underbrace{1,\ldots,1}_{i-1+e_1'},2,\underbrace{1,\ldots,1}_{r-i+e_2'})\\
&=\sum_{\substack{e_1'+e_2'=k-r-1\\^\forall e_j'\geq0}}(-1)^{r-i+e_2'+1}\binom{k}{r-i+e_2'+1}\left(\frac{B_{p-k}}{k}\bmod p\right)_p \quad(\textnormal{by \cite[Theorem 4.5]{hes}})\\
&=\sum_{e=r-i+1}^{k-i}(-1)^{e}\binom{k}{e}\left(\frac{B_{p-k}}{k}\bmod p\right)_p 
\quad(\textnormal{by putting}\;\:e=r-i+e_2'+1)\\
&=\left(\sum_{e=0}^{k-i}-\sum_{e=0}^{r-i}\right)(-1)^{e}\binom{k}{e}\left(\frac{B_{p-k}}{k}\bmod p\right)_p\\
&=(-1)^{i-1}\left((-1)^{k+1}\binom{k-1}{i-1}+(-1)^{r}\binom{k-1}{r-i}\right)\left(\frac{B_{p-k}}{k}\bmod p\right)_p.
\end{align*}
Thus we have
\begin{equation}\label{sum}
\sum_{\substack{k_1+\cdots+k_r=k\\^\forall k_j\geq1,k_i\geq2}}\za(k_1,\ldots,k_r)=
(-1)^{i-1}\left((-1)^{k+1}\binom{k-1}{i-1}+(-1)^{r}\binom{k-1}{r-i}\right)\left(\frac{B_{p-k}}{k}\bmod p\right)_p.
\end{equation}
If $k$ is even, then the right-hand side of \eqref{sumformula} is equal to the right-hand side of \eqref{sum} because $B_{p-k}=0$ whenever $p\geq k+3$. 
If $k$ is odd, then the right-hand side of the \eqref{sumformula} is equal to the right-hand side of \eqref{sum} because $(-1)^{k+1}=1$. 
Therefore the proof is complete.
\end{proof}


\section*{Acknowledgment}
The author would like to express his sincere gratitude to Professors Masanobu Kaneko and Shingo Saito for helpful comments and advice.

\noindent
{Kojiro Oyama}\\
{\it E-mail}: {\tt k-oyama@kyudai.jp}


\begin{thebibliography}{99}
 \bibitem{hoffman2}M. E. Hoffman, Multiple harmonic series, Pacific J. Math. \textbf{152} (1992), 275--290.
 \bibitem{hoffman}M.~E.~Hoffman, Quasi-symmetric functions and mod \textit{p} multiple harmonic sums, Kyushu J. Math. \textbf{69} (2015), 345-366.
 \bibitem{hoffman3}M. E. Hoffman, The algebra of multiple harmonic series, J. Algebra \textbf{194} (1997), 477--495.
 \bibitem{ikz}K. Ihara, M. Kaneko and D. Zagier, Derivation and double shuffle relations for multiple zeta values,
  Compositio Math. \textbf{142} (2006), 307-338.
 \bibitem{kaneko}M. Kaneko, Finite multiple zeta values. (in Japanese), RIMS K$\hat{\textnormal{o}}$ky$\hat{\textnormal{u}}$roku Bessatsu, to appear.
 \bibitem{kaneko2}M. Kaneko and D. Zagier, Finite multiple zeta values, in preparation.
 \bibitem{hes}Kh. Hessami Pilehrood, T. Hessami Pilehrood, and R. Tauraso, New properties of multiple harmonic sums modulo \textit{p} and \textit{p}-analogues of Leshchiner's series, Trans. Amer. Math. Soc. \textbf{366} (2014), 3131-3159.
 \bibitem{reu}C. Reutenauer, \textit{Free Lie Algebras}, Oxford Science Publications, 1993.
 \bibitem{saito}S.~Saito and N.~Wakabayashi, Sum formula for finite multiple zeta values, J. Math. Soc. Japan \textbf{67} (2015), 1069-1076.
 \bibitem{saito2}S. Saito and N. Wakabayashi, Bowman-Bradley type theorem for finite multiple zeta values, Tohoku Math. J. (2) \textbf{68} (2016), 241-251.
\end{thebibliography}
\end{document}